\documentclass[11pt, oneside]{article}
\usepackage[utf8]{inputenc}
\usepackage{amsmath}
\usepackage{amsfonts}
\usepackage{amssymb}
\usepackage{amsthm}
\newtheorem{theorem}{Theorem}[section]

\newtheorem{example}[theorem]{Example}

\newtheorem{remark}{\sc Remark}
\newtheorem{lemma}{\sc Lemma}[section]
\newtheorem{corollary}{\sc Corollary}[section]
\newtheorem{definition}{\sc Definition}[section]

\newcommand{\be}{\begin{eqnarray}}
\newcommand{\ee}{\end{eqnarray}}
\newcommand{\Be}{\begin{eqnarray*}}
\newcommand{\Ee}{\end{eqnarray*}}
\newcommand{\bee}{\begin{equation}}
\newcommand{\eee}{\end{equation}}
\newcommand{\ba}{\begin{array}}
\newcommand{\ea}{\end{array}}
\newcommand{\bl}{\begin{lemma}}
\newcommand{\el}{\end{lemma}}
\newcommand{\bd}{\begin{definition}}
\newcommand{\ed}{\end{definition}}
\newcommand{\bt}{\begin{theorem}}
\newcommand{\et}{\end{theorem}}
\newcommand{\bp}{\begin{proof}}
\newcommand{\ep}{\end{proof}}
\newcommand{\bi}{\begin{itemize}}
\newcommand{\ei}{\end{itemize}}
\newcommand{\br}{\begin{remark}}
\newcommand{\er}{\end{remark}}
\newcommand{\bc}{\begin{corollary}}
\newcommand{\ec}{\end{corollary}}
\newcommand{\bex}{\begin{example}}
\newcommand{\eex}{\end{example}}
\begin{document}
\date{}
\title{\textbf{Projective changes between two Finsler spaces with $(\alpha, \beta )$-metrics} }
\maketitle
\begin{center} 
\author{\textbf{Gauree Shanker* and Sruthy Asha Baby** }}\\
*Department of Mathematics and Statistics\\
Central University of Punjab, Bathinda\
Punjab-151 001, India\\
Email:gshankar@cup.ac.in\\
**Department of Mathematics and Statistics\\
Banasthali University, Banasthali\\
Rajasthan-304 022, India\\
Email:sruthymuthu123@gmail.com.
\end{center}
\begin{center}
\textbf{Abstract}
\end{center}
\begin{small}
In the present paper, we find the conditions to characterize projective change between two $(\alpha, \beta)$-metrics, F = $\alpha + \epsilon \beta + k\frac{\beta^2}{\alpha}$ ($\epsilon$ and k $\neq$ 0 are constants) and a Matsumoto metric $\bar{F}=\frac{\bar{\alpha}^2}{\bar{\alpha} -\bar{\beta}}$ on a manifold with dimension $n \geq 3$ where $\alpha$ and $\bar{\alpha}$ are two Riemannian metrics, $\beta$ and $\bar{\beta}$ are two non-zero 1-forms. Moreover, we study such projective changes when F has some special curvature properties.\\
\end{small}

\section{Introduction}

The projective change between two Finsler spaces have been studied by many authors ( [\ref{PC-2}], [\ref{PC-5}], [\ref{PC-6}], [\ref{PC-14}] ). An interesting result concerned with the theory of projective change was given by Rapcsak's paper [\ref{PC-11}]. He proved the necessary and sufficient condition for projective change. In 2008, H. S. Park and Y. Lee [\ref{PC-8}] studied on a class of projectively changes between a Finsler space with $(\alpha, \beta)$-metric and the associated Riemannian metric.\\

In this paper, we find the relation between two Finsler spaces with general $(\alpha, \beta)$-metric F = $\alpha + \epsilon \beta + k\frac{\beta^2}{\alpha}$ ($\epsilon$ and k $\neq$ 0 are constants) and a Matsumoto metric $\bar{F}=\frac{\bar{\alpha}^2}{\bar{\alpha} -\bar{\beta}}$ respectively under projective change.\\

\section{Preliminaries}

The terminology and notation are referred to ( [\ref{PC-1}], [\ref{PC-3}], [\ref{PC-12}] ). Let $F^n = (M, F)$ be a Finsler space on a differential manifold M endowed with a fundamental function $F(x, y)$. We use the following notations:\\
\begin{eqnarray*}
(a). &&g_{ij} = \frac{1}{2} \dot{\partial}_i \dot{\partial}_j F^2, \ \ \dot{\partial}_j = \frac{\partial}{\partial y^i},\\
(b). &&C_{ijk} = \frac{1}{2} \dot{\partial}_kg_{ij},\\
(c). &&h_{ij} = g_{ij} - l_il_j,\\
(d). &&\gamma^i_{jk} = \frac{1}{2} g^{ir} (\partial_j g_{rk} + \partial_k g_{rj} - \partial_r g_{jk}),\\
(e). &&G^i = \frac{1}{2} \gamma ^i_{jk} y^j y^k,\ \  G^i_j = \dot{\partial}_j G^i,\ \ G^i_{jk} = \dot{\partial_k} G^i_j,\ \ G^i_{jkl} = \dot{\partial_j} G^i_{jk}.
\end{eqnarray*}

The concept of $(\alpha, \beta)$-metric $F(\alpha, \beta)$ was introduced in 1972 by M. Matsumoto and studied by many authors ( [\ref{PC-10}], [\ref{PC-15}]).
The Finsler space $F^n = (M, F)$ is said to have an $(\alpha, \beta)$-metric, if F is a positively homogenous function of a Finsler metric on a same underlying manifold M and it is called projective if any geodesic in $(M, F)$ remains to be a geodesic in $(M, \bar{F})$ and viceversa. We say that a Finsler metric is projectively related to another metric if they have the same geodesics as point sets. In Riemannian geometry, two Riemannian metrics $\alpha$ and $\bar{\alpha}$ are projectively related if and only if, their spray coefficients have the relation [\ref{PC-5}]

\begin{equation}
G^i_\alpha = G^i_{\bar{\alpha}}+ \lambda_{x^k} y^k y^i,
\end{equation}
where $\lambda = \lambda(x)$ is a scalar function on the based manifold and $(x^i, y^i)$ denotes the local coordinates in the tangent bundle TM.\\

Two Finsler metrics $F$ and $\bar{F}$ are projectively related if and only if their spray coefficients have the relation [\ref{PC-5}]\\

\begin{equation}\label{PC.10.10}
G^i = \bar{G}^i + P(y)y^i,
\end{equation}
where P(y) is a scalar function on $TM \backslash \{0\}$ and homogenous of degree one in $y$. A Finsler metric is called a projectively flat metric, if it is projectively related to a locally Minkowskian metric.\\

For a given Finsler metric $F = F(x,y)$, the geodesics of F satisfy the following ODEs:\\

\begin{equation}
\frac{d^2 x^i}{dt^2} + 2 G^i (x, \frac{dx}{dt}) = 0,
\end{equation}
where $G^i = G^i (x,y)$ are called the geodesic coefficients, which are given by 

\begin{equation}
G^i = \frac{1}{4} g^{il} \{[F^2]_{x^m y^l}y^m -[F^2]_{x^l} \}.
\end{equation}
\\
Let $\phi = \phi(s)$, $|s| < b_0$ be a positive $C^{\infty}$ function satisfying the following 

\begin{equation}\label{P.2.3}
\phi(s) - s \phi'(s) + (b^2 -s^2)\phi''(s) > 0,\ \ \ \ (|s| \leq b < b_0).
\end{equation}
\\
If $\alpha = \sqrt{a_{ij} y^i y^j}$ is a Riemannian metric and $\beta = b_iy^i$ is a 1-form satisfying \textbardbl $\beta_x$ \textbardbl $_{\alpha} < b_0$ $\forall$ x $\in$ M, then F = $\alpha \phi (s)$, s = $\frac{\beta}{\alpha}$, is called an (regular ) $(\alpha , \beta)$-metric. In this case, the fundamental form of the metric tensor induced by F is positive definite.\\
\\
Let $\bigtriangledown \beta $ = $b_{i|j} dx^i \otimes dx^j$ be a covarient derivative of $\beta$ with respect to $\alpha$.\\
\\
Denote
\begin{equation}
r_{ij} = \frac{1}{2} (b_{i|j} + b_{j|i}),\ \ \ \ \ s_{ij} = \frac{1}{2}(b_{i|j} - b_{j|i}),
\end{equation}
\\
$\beta$ is closed  if and only if $s_{ij} = 0$ [\ref{PC-13}]. Let $s_j = b^i s_{ij}$, $s^i_j$ = $a^{il}s_{lj}$, $s_0$= $s_i y^i$, $s^i_0 = s^i_j y^j$ and $r_{00} = r_{ij} y^i y^j$.\\

The relation between the geodesic coefficients $G^i$ of F and geodesic coefficients $G^i_{\alpha}$ of $\alpha$ is given by\\
\begin{equation}\label{P.2.4}
G^i = G^i_{\alpha} + \alpha Q s^i_0+ \{-2 Q \alpha s_0 + r_{00}\}\{ \Psi b^i + \Theta \alpha^{-1} y^i \},
\end{equation}
where
\begin{eqnarray*}
\Theta &=& \frac{\phi \phi' - s (\phi\phi'' +\phi'\phi')}{2\phi\{(\phi -s\phi') + (b^2 -s^2)\phi''\}},\\
Q &=& \frac{\phi'}{\phi -s \phi'},\\
\Psi &=& \frac{1}{2} \frac{\phi''}{\{ (\phi -s\phi') + (b^2 -s^2)\phi''\}}.
\end{eqnarray*}
\begin{definition}(Douglas metric)
Let 
\begin{equation}\label{P.2.5}
D^i_{jkl} = \frac{\partial^3}{\partial y^j \partial y^k \partial y^l}\Bigl( G^i - \frac{1}{n+1} \frac{\partial G^m}{\partial y^m} y^i\Bigr),
\end{equation}
where $G^i$ are the spray coefficients of F. The tensor $D = D^i_{jkl} \partial_i \otimes dx^j \otimes dx^k \otimes dx^l$ is called the Douglas tensor. A Finsler metric is called Douglas metric if the Douglas tensor vanishes.\\
\end{definition}
In [\ref{PC-7}], the authors characterized the $(\alpha, \beta)$-metrics of Douglas type.
\begin{lemma} ([\ref{PC-7}])
Let F = $\alpha \phi (\frac{\beta }{\alpha})$ be a regular $(\alpha, \beta)$-metric on an n-dimensional
manifold M $(n \geq 3)$. Assume that $\beta$ is not parallel with respect to $\alpha$ and $db\neq 0$
everywhere or b=constant, and F is not of Randers type. Then F is a Douglas metric
if and only if the function $\phi = \phi(s)$ with $\phi (0) =1$ satisfies following ODE:
\begin{equation}
[1 + (k_1 + k_2s^2)s^2 + k_3s^2]\phi'' = (k_1 + k_2s^2)(\phi - s \phi')
\end{equation}
and $\beta$ satisfies 
\begin{equation}
b_{i|j} = 2\sigma[(1 + k_1b^2)a_{ij} + (k_2b^2 + k_3)b_ib_j ];
\end{equation}
where $b^2$ =\textbardbl $\beta$ \textbardbl $^2_{\alpha}$ and $\sigma = \sigma (x)$ is a scalar function and $k_1$, $k_2$, $k_3$ are constants with $(k_2, k_3)\neq 0$.
\end{lemma}

\begin{lemma}([\ref{PC-8}])
A Matsumoto space $F^n = (M, F) $ with $F$= $\frac{\alpha^2}{\alpha - \beta}$ $(n>2)$ is of Douglas type, if and only if $b_{i|j}=0$.
\end{lemma}

We know that Douglas tensor is a projective invariant [\ref{PC-7}]. Note that the spray coefficients of a Riemannian metric are quadratic forms and one can see that the Douglas tensor vanishes from (\ref{P.2.5}). This shows that Douglas tensor is a non-Riemannian quality.\\

In the following, we use quantities with bar to denote the corresponding quantities of the metric $\bar{F}$. Now, we compute the Douglas tensor of a general $(\alpha, \beta)$-metric.\\
Let 
\begin{equation}
\bar{G}^i = G^i_{\alpha} + \alpha Q s^i_0 +\Psi \{-2 Q \alpha s_0 + r_{00}\} b^i.
\end{equation}
Then (\ref{P.2.4}) becomes 
\begin{equation}
G^i = \bar{G}^i + \Theta \{ -2 Q \alpha s_0 + r_{00}\} \alpha^{-1} y^i.
\end{equation}
Clearly, $G^i$ and $\bar{G}^i$ are projective equivalent according to (\ref{PC.10.10}), they have the same Douglas tensor.\\
Let 
\begin{equation}\label{P.2.6}
T^i = \alpha Q s^i_0 + \Psi \{-2 Q \alpha s_0 + r_{00}\}b^i.
\end{equation}
Then $\bar{G}^i$ = $G^i_{\alpha } + T^i$, thus 
\begin{eqnarray*}
D^i_{jkl} &=& \bar{D}^i_{jkl},\\
&=& \frac{\partial^3}{\partial y^j \partial y^k \partial y^l}\Bigl( G^i_{\alpha} - \frac{1}{n + 1} \frac{\partial G^m_{\alpha}}{\partial y^m} y^i + T^i - \frac{1}{n + 1} \frac{\partial T^m}{\partial y^m} y^i \Bigr),\\
&=&  \frac{\partial^3}{\partial y^j \partial y^k \partial y^l} \Bigl(T^i  -\frac{1}{n + 1} \frac{\partial T^m}{\partial y^m} y^i \Bigr).
\end{eqnarray*}
\begin{equation}
\label{P.2.7}
\end{equation}
To simplify (\ref{P.2.7}), we use the following identities
\begin{equation}
\alpha_{y^k} = \alpha^{-1} y_k, \ \ \ s_{y^k} = \alpha^{-2} (b_k \alpha - s y_k),
\end{equation}
where $y_i = a_{il}y^l$, $\alpha_{y^k}= \frac{\partial \alpha}{\partial y^k}$. Then
\begin{eqnarray*}
[\alpha Q s^m_0]_{y^m} &=& \alpha^{-1} y_m Q s^m_0 + \alpha^{-2} Q' [b_m \alpha^2 - \beta y_m]s^m_0,\\
&=& Q' s_0
\end{eqnarray*}
and 
\begin{eqnarray*}
[\Psi (-2 Q) \alpha s_0 + r_{00}b^m]_{y^m} &=& \Psi' \alpha^{-1} (b^2 -s^2 )[r_{00} - 2 Q \alpha s_0]\\&& + 2 \Psi [r_0 - Q' (b^2 - s^2) s_0 - Q s s_0],
\end{eqnarray*}
where $r_j = b^i r_{ij}$ and $r_0 = r_i y^i$. Thus from (\ref{P.2.6}), we obtain

\begin{eqnarray*}
T^m_{y^m} &=& Q' s_0 + \Psi' \alpha^{-1} (b^2- s^2) [r_{00} -2 Q \alpha s_0] \\&& + 2 \Psi [r_0 - Q' (b^2 - s^2) s_0 - Q s s_0],
\end{eqnarray*}
\begin{equation}
\label{P.2.8}
\end{equation}
Now, we assume that the $(\alpha, \beta)$-metrics $F$ and $\bar{F}$ have the same Douglas tensor, i.e., $D^i_{jkl} = \bar{D}^i_{jkl}$. Thus from (\ref{P.2.5}) and (\ref{P.2.7}), we get

\begin{equation*}
\frac{\partial^3}{\partial y^j \partial y^k \partial y^l} \Bigl(T^i -\bar{T}^i -\frac{1}{n + 1} (T^m_{y^m}- \bar{T}^m_{y^m} )y^i\Bigr) =0.
\end{equation*}
Then there exists a class of scalar function $H^i_{jk} = H^i_{jk}(x)$, such that
\begin{equation}\label{P.2.9}
H^i_{00} = T^i -\bar{T}^i -\frac{1}{n + 1} (T^m_{y^m}- \bar{T}^m_{y^m} )y^i,
\end{equation}
where $H^i_{00} = H^i_{jk}y^j y^k$, $T^i$ and $T^m_{y^m}$ are given by the relations  (\ref{P.2.6}) and (\ref{P.2.8}) respectively.

\section{Projective change between two Finsler spaces with $(\alpha, \beta)$-metric}
In this section, we find the projective relation between two important class of  $(\alpha, \beta)$-metrics, general $(\alpha, \beta)$-metric $\alpha + \epsilon \beta + k \frac{\beta^2}{\alpha}$ and Matsumoto metric $\frac{\bar{\alpha}^2}{\bar{\alpha} - \bar{\beta}}$ on a same underlying manifold M of dimension $n \geq 3$. \\

Now, let us consider the general $(\alpha, \beta)$-metrics in the form $F = \alpha + \epsilon \beta + k\frac{\beta^2}{\alpha}$. Let $b_0 = b_0(k) > 0$
be the largest number such that $1 + 2kb^2 + 3ks^2 > 0$, $|s|\leq b < b_0$. Then $F = \alpha + \epsilon \beta + k\frac{\beta^2}{\alpha}$ is a Finsler metric if and only if $\beta$ satisfies b =\textbardbl $\beta$ \textbardbl$_{\alpha}< b_0$. By (\ref{P.2.4}), the geodesic coefficients of F are determined by

\begin{eqnarray*}
Q &=& \frac{\epsilon + 2ks}{1-ks^2},\\
\Theta &=& \frac{\epsilon -3\epsilon k s^2 - 4k^2 s^2}{2(1+2kb^2 -3ks^2)(1+\epsilon s +ks^2)},\\
\Psi &=& \frac{k}{1 + 2kb^2 -3ks^2}.
\end{eqnarray*}
\begin{equation}
\label{P.3.1}
\end{equation}
Now, for Matsumoto metric $\bar{F}$ =$\frac{\bar{\alpha}^2}{\bar{\alpha} -\bar{\beta}}$, one can prove by
(\ref{P.2.3}) that F is a regular Finsler metric if and only if 1-form satisfies the condition \textbardbl $\bar{\beta}$ \textbardbl$_{\bar{\alpha}} <\frac{1}{2}$, for any $x\in M$. The geodesic coefficients are given by (\ref{P.2.4}) with
\begin{eqnarray*}
\bar{Q} &=& \frac{1}{1-2s},\\
\bar{\Theta} &=& \frac{1-4s}{2(1  + 2b^2-3s)},\\
\bar{\Psi} &=& \frac{1}{1+ 2b^2 -3s}.
\end{eqnarray*}
\begin{equation}
\label{P.3.2}
\end{equation}
First, we have the following lemma:
\begin{lemma}\label{P.lem.1}
Let $F= \alpha + \epsilon \beta + k\frac{\beta^2}{\alpha}$ ($\epsilon$ and $k\neq 0$ are constants ) be a general $(\alpha, \beta)$-metric and $\bar{F}$ = $\frac{\bar{\alpha}^2}{\bar{\alpha} - \bar{\beta}}$ be a Matsumoto metric on a manifold M with dimension $n\geq 3$, where $\alpha$ and $\bar{\alpha}$ are two Riemannian metrics, $\beta$ and $\bar{\beta}$ are two nonzero collinear 1-forms. Then $F$ and $\bar{F}$ have the same Douglas tensor if and only if they are all Douglas metrics.
\end{lemma}
\begin{proof}
First, we prove the sufficient condition. Let F and $\bar{F}$ be Douglas metrics and corresponding Douglas tensors be $D^i_{jkl}$ and $\bar{D}^i_{jkl}$. Then by the definition of Douglas metric, we have $D^i_{jkl} =0$ and $\bar{D}^i_{jkl} = 0$, i.e., both F and $\bar{F}$ have the same Douglas tensor.\\
Next, we have the necessary condition. If F and $\bar{F}$ have the same Douglas tensor, then (\ref{P.2.9}) holds. Substituting (\ref{P.3.1}) and (\ref{P.3.2}) in (\ref{P.2.9}), we obtain
\begin{eqnarray*}
H^i_{00} &=& \frac{A^i \alpha^9 + B^i \alpha^8 + C^i \alpha^7 + D^i \alpha^6 + E^i \alpha^5 + F^i \alpha^4 + H^i \alpha^3 + P^i \alpha^2 + Q^i}{I\alpha^8 + J\alpha^6 + K\alpha^4 + L\alpha^2 + M }-\\&&\frac{\bar{A}^i \bar{\alpha}^6 + \bar{B}^i \bar{\alpha}^5 + \bar{C}^i \bar{\alpha}^4 + \bar{D}^i \bar{\alpha}^3 + \bar{E}^i \bar{\alpha}^2 + \bar{F}^i \bar{\alpha} + \bar{H}^i}{\bar{I}\bar{\alpha}^5 + \bar{J}\bar{\alpha}^4 + \bar{K}\bar{\alpha}^3 + \bar{L}\bar{\alpha}^2 + \bar{M} \bar{\alpha}},
\end{eqnarray*}
\begin{equation}
\label{P.3.4}
\end{equation}
where
\begin{eqnarray*}
A^i&=& (1 + 2kb^2)\{\epsilon (1 + 2kb^2) s^i_0 - 2 \epsilon k s_0 b^i\},\\
B^i&=& (1 +2kb^2 )\{ 2k ^2 (1 +2kb^2 )\beta s^i_0 - 2k( 2k b^i \beta + \mu y^i)s_0\\&& - 2 \mu k  y^i r_0 - k r_{00} b^i \},\\
C^i&=& \bigl\{ -( 7 + 2 k b^2) (1 + 2k b^2 )\epsilon k s^i_0 \beta^2 + 4 \epsilon ( 2+ kb^2) k^2 b^i \beta^2 s_0 - 12\mu \epsilon k^2 \beta s_0 b^2 y^i \bigr\},\\
D^i&=& (-14 - 4kb^2 ) (1 + 2k b^2) k^2 \beta^3 s^i_0 + 8 k^3 (2 + k b^2) b^i \beta^3 s_0 + (3k^2 b^i \beta^2 \\&&- 6 \mu k^2 b^2 \beta y^i)r_{00} + \mu k^2 (10 + 8 kb^2) r_0 y^i \beta^2 + \mu k^2 (10 + 32 k b^2) \beta^2 s_0 y^i ,\\
E^i &=& - 8 \epsilon k^2 \beta^4 s^i_0 - 6\epsilon k^3 \beta^4 s_0 b^i - 12 \mu \epsilon k^2 (1 + kb^2) \beta^3 s_0,\\
F^i &=& k^2 \beta^3\{ 18 k(1 + kb^2) \beta^2 s^i_0 -\mu k(14 + 4 k b^2) \beta r_0 y^i\\&& -12k \beta (4y^i + 5 k b^i \beta)s_0 - (1+2kb^2) (kb^i\beta + 6\mu y^i) r_{00}\} ,\\
H^i &=& 3 \epsilon k^3 \beta^5 \{ 3 \beta s^i_0 + 4 \mu s_0 y^i\},\\
P^i &=& 3 k^3 \beta^5 \Bigl[ - \bigl\{ kb^i \beta + 2 \mu (2 + kb^2) y^i\bigr\} r_{00} + 2 k\beta \bigl\{-3s^i_0 \beta + \mu (5s_0 + r_0) y^i \bigr\} \Bigr],\\
Q^i &=& 6 \mu k^4 \beta^7 r_{00} y^i,
\end{eqnarray*}
\begin{eqnarray*}
\bar{A}^i &=& -(1 + 2\bar{b}^2) \{2\bar{b}^i \bar{s}_0 - (1+2\bar{b}^2) \bar{s}^i_0 \},\\
\bar{B}^i &=& (1+ 2\bar{b}^2) [- 4\bar{\beta} (2 + b^2)\bar{s}^i_0 + \bar{b}^i \bar{r}_{00} - 2 \mu y^i\{(1+ 2\bar{b}^2)\bar{s}_0 + \bar{r}_0\}] \\&&+ 2(5+ 4\bar{b}^2) \{\bar{b}^2 \mu y^i + \bar{b}^i \beta\}\bar{s}_0,\\
\bar{C}^i &=& 2\bar{\beta} (1+2\bar{b}^2) \{2 (3 \bar{\beta} \bar{s}^i_0- \bar{b}^i\bar{r}_{00}) + \mu y^i (7 \bar{s}_0 + 4\bar{r}_0)\} + 3[3 \bar{\beta}^2 \bar{s}^i_0 \\&&- \mu y^i \{\bar{b}^2 \bar{r}_{00} + 2 \bar{\beta} (4\bar{b}^2 \bar{s}_0 - \bar{r}_0)\}],\\
\bar{D}^i &=& -2 \bar{\beta} \{19 \bar{\beta}^2 \bar{s}^i_0 - 8 \bar{b}^i \bar{\beta}(\bar{b}^2 + 2) \bar{r}_{00} + 2\mu y^i(19 \bar{\beta} \bar{s}_0 + \\&& 24 \bar{\beta} \bar{r}_0 + 8 \bar{b}^2 \bar{\beta} \bar{s}_0 - 6\bar{b}^2 \bar{r}_{00})\},\\
\bar{E}^i &=& -3\bar{\beta}^2 [4\bar{b}^i \bar{\beta} \bar{r}_{00} + \mu y^i \{(4\bar{b}^2-1) \bar{r}_{00} - 4 \bar{\beta} (3\bar{s}_0 + 2\bar{r}_0)\}],\\
\bar{F}^i &=& -12\mu y^i \bar{\beta}^3 \bar{r}_{00},\\
\bar{H}^i &=& 12\mu y^i \bar{\beta}^4 \bar{r}_{00},\\
\mu &=& \frac{1}{ n + 1}.
\end{eqnarray*}
\begin{equation}
\label{P.3.5}
\end{equation}
and
\begin{eqnarray*}
I &=& (2kb^2 + 1)^2,\\
J &=& -4k (1 + 2kb^2) (2 + kb^2) \beta^2,\\
K &=& k^2 \beta^4 (22 + 38 kb^2 + 4k^2 b^4),\\
L &=& -12k^3 \beta^6 (b^2 k + 2 ),\\
M &=& 9k^4 \beta^8,
\end{eqnarray*}
\begin{eqnarray*}
\bar{I} &=& (1 +2\bar{b}^2 )^2,\\
\bar{J} &=& -2 \bar{\beta} \{5 + 2\bar{b}^2 (7 + 4\bar{b}^2)\},\\
\bar{K} &=& \bar{\beta}^2 \{ 37 + 16\bar{b}^2 (\bar{b}^2 + 4)\},\\
\bar{L} &=& -12\bar{\beta}^3 (4 \bar{b}^2 + 5),\\
\bar{M} &=& 36 \bar{\beta}^4.
\end{eqnarray*}
\begin{equation}
\label{P.3.6}
\end{equation}
Then (\ref{P.3.4}) is equivalent to 
\begin{eqnarray*}
&&H^i_{00} (I\alpha^8 + J\alpha^6 + K\alpha^4 + L\alpha^2 + M)(\bar{I}\bar{\alpha}^5 + \bar{J}\bar{\alpha}^4 + \bar{K}\bar{\alpha}^3 + \bar{L}\bar{\alpha}^2 + \bar{M}\bar{\alpha}) \\&&= (A^i \alpha^9 + B^i \alpha^8 + C^i \alpha^7 + D^i \alpha^6 + E^i \alpha^5 + F^i \alpha^4 + H^i\alpha^3 + P^i \alpha^2 + Q^i)\\&&
(\bar{I}\bar{\alpha}^5 + \bar{J}\bar{\alpha}^4 + \bar{K}\bar{\alpha}^3 + \bar{L}\bar{\alpha}^2 + \bar{M}\bar{\alpha}) -  (\bar{A}^i \bar{\alpha}^6 + \bar{B}^i \bar{\alpha}^5 + \bar{C}^i \bar{\alpha}^4 + \bar{D}^i \bar{\alpha}^3 + \bar{E}^i \bar{\alpha}^2 +\\&& \bar{F}^i \bar{\alpha} + \bar{H}^i)(I\alpha^8 + J\alpha^6 + K\alpha^4 + L\alpha^2 + M),
\end{eqnarray*}
or
\begin{eqnarray*}
&&H^i_{00} l (I\alpha^8 + J\alpha^6 + K\alpha^4 + L\alpha^2 + M) = l(A^i \alpha^9 + B^i \alpha^8 + C^i \alpha^7 + D^i \alpha^6 +\\&& E^i \alpha^5 + F^i \alpha^4 + H^i\alpha^3 + P^i \alpha^2 + Q^i) - m (I\alpha^8 + J\alpha^6 + K\alpha^4 + L\alpha^2 + M),
\end{eqnarray*}
where
\begin{eqnarray}
l &=& (\bar{I}\bar{\alpha}^5 + \bar{J}\bar{\alpha}^4 + \bar{K}\bar{\alpha}^3 + \bar{M}\bar{\alpha}),\\
m &=& (\bar{A}^i \bar{\alpha}^6 + \bar{B}^i \bar{\alpha}^5 + \bar{C}^i \bar{\alpha}^4 + \bar{D}^i \bar{\alpha}^3 + \bar{E}^i \bar{\alpha}^2 + \bar{F}^i \bar{\alpha} + \bar{H}^i).
\end{eqnarray}
This can be written as
\begin{eqnarray*}
&& lA^i\alpha^9 + (l B^i -mI - H^i_{00} l I)\alpha^8 + l C^i \alpha^7 + (lD^i -m J- H^i_{00} l J) \alpha^6\\&& + lE^i\alpha^5 + (l F^i -mK - H^i_{00} l K)\alpha^4 + l H^i \alpha^3 + (l P^i -mL - H^i_{00} lL)\alpha^2\\&&+ (l Q^i -m M - H^i_{00} l M) = 0,
\end{eqnarray*}
which implies
\begin{eqnarray*}
&& lA^i\alpha^9 + l C^i \alpha^7+ lE^i\alpha^5 + l H^i \alpha^3 =- (l B^i -mI - H^i_{00} l I)\alpha^8  \\&&- (lD^i -m J- H^i_{00} l J) \alpha^6  - (l F^i -mK - H^i_{00} l K)\alpha^4  - (l P^i -mL - H^i_{00} lL)\alpha^2\\&&- (l Q^i -m M - H^i_{00} l M).
\end{eqnarray*}
\begin{equation}
\label{P.3.10}
\end{equation}
Since $l\ =\ \bar{\alpha}(\bar{I}\bar{\alpha}^4 + \bar{J}\bar{\alpha}^3 + \bar{K}\bar{\alpha}^2 + \bar{M})$, 
from (\ref{P.3.10}), we can see that R. H. S. has a factor of $\bar{\alpha}$. Here, $m M$ is the only term not containing $\bar{\alpha}$. This implies that $\bar{H}^i M$ must be a factor of $\bar{\alpha}$ i.e., $12 \mu\bar{y}^i\bar{\beta}^4 \bar{r}_{00}$ has a factor of $\bar{\alpha}$. Since $\bar{\beta}^2$ has no factor of $\bar{\alpha}$, the only possibility is that $\bar{\beta}\bar{r}_{00}$ has the factor $\bar{\alpha}^2$. Then for each i there exists a scalar function $\tau^i = \tau (x)$ such that $\bar{\beta}= \tau^i \alpha^2$ which is equivalent to\\
$$b_j r_{0k} + 2b_k r_{j0} = 2 \tau^i a_{jk}.$$
When n $<$ 2 and we assume that $\tau^i \neq 0$, then  
\begin{eqnarray*}
2 &\geq& rank(b_j r_{0k}) + rank(b_k r_{0j})\\
&>& rank(b_j r_{0k} + b_k r_{0j})\\
&=& rank (2 \tau^i a_{jk}) > 2,
\end{eqnarray*}
\begin{equation}
\label{P.3.11}
\end{equation}
which is impossible unless $\tau^i = 0$. Then $\bar{\beta}\bar{r}_{00} = 0$. Since $\bar{\beta} \neq 0$, we have $\bar{r}_{00} = 0$ implies $\bar{b}_{i|j} = 0$ i. e., $\bar{\beta}$ is closed.\\
It is well known that Matsumoto metric $\bar{F} = \frac{\bar{\alpha}^2}{\bar{\alpha} - \bar{\beta}}$ is a Douglas metric if and only if $\bar{b}_{i|j} = 0$. So $\bar{F}$ is a Douglas metric. Since F and $\bar{F}$ have the same Douglas tensor, they are Douglas metrics. Hence the Lemma is proved completely.
\end{proof}
Now, we prove the following main theorem:

\begin{theorem}\label{P.them.1}
The Finsler metric F = $\alpha + \epsilon \beta+ k \frac{\beta^2}{\alpha}$ ($\epsilon$ and k $\neq 0$ are constants) is projectively related to $\bar{F}= \frac{\bar{\alpha}^2}{\bar{\alpha} -\bar{\beta}}$ if and only if the following conditions are satisfied.
\begin{eqnarray*}
G^i_\alpha &=& G^i_{\bar{\alpha}} + \theta y^i - 2k \tau \alpha^2 b^i,\\
b_{i|j} &=& 2\tau \{(1 + 2kb^2) a_{ij} - 3k b_ib_j\},\\
d\bar{\beta} &=& 0\ or\ \bar{b}_{i|j} =0,
\end{eqnarray*}
\begin{equation}
\label{P.given}
\end{equation}
where b = $\textbardbl \beta \textbardbl_{\alpha}$, $\tau = \tau(x)$ is a scalar function, $\theta = \theta_i y^i$ is a 1-form on M and $b_{i|j}$ denote the coefficients of the covariant derivatives of $\beta$ with respect to $\alpha$.
\end{theorem}

\begin{proof}
We first show the necessity. Since F is projectively equivalent to $\bar{F}$, they have the same Douglas tensor. By Lemma \ref{P.lem.1}, we obtain that F and $\bar{F}$ are Douglas metrics. By [\ref{PC-7}], we know that the $(\alpha, \beta )$- metric F= $\alpha + \epsilon \beta + k \frac{\beta^2}{\alpha}$ is a Douglas metric if and only if
\begin{eqnarray}\label{P.3.14}
b_{i|j} = 2 \tau \{(1+ 2kb^2) a_{ij} - 3 k b_ib_j\},
\end{eqnarray}
where $\tau = \tau (x)$ is a scalar function on M. In this case, $\beta$ is closed. Plugging (\ref{P.3.14}) and  (\ref{P.3.1}) into (\ref{P.2.4}) yields
\begin{eqnarray}\label{P.3.15}
G^i = G^i_\alpha + \frac{(\epsilon \alpha^3 - 3 \epsilon k \alpha \beta^2 - 4k^2 \beta^3)}{\alpha^2 + \epsilon \alpha \beta + k\beta^2}\tau y^i + 2k\tau \alpha^2 b^i,
\end{eqnarray}
On the other hand, plugging (\ref{P.3.2}) into (\ref{P.2.4}), we get
\begin{equation}\label{P.3.16}
\bar{G}^i = G^i_{\bar{\alpha}} + \frac{\bar{\alpha}^2 s^i_0}{\bar{\alpha} - 2 \bar{\beta}} + \Bigl[\frac{-2\bar{\alpha}^2 \bar{s}_0 }{\bar{\alpha} - 2 \bar{\beta} + \bar{r}_{00} }\Bigr] \Bigl[\frac{2 \bar{\alpha}^2 \bar{b}^i + (\bar{\alpha} - 4 \bar{\beta})y^i}{ 2 \bar{\alpha} (\bar{\alpha} + 2 \bar{\alpha} \bar{b}^2 - 3 \bar{\beta} )}\Bigr].
\end{equation}
By the projective equivalence of F and $\bar{F}$ again, there is a scalar function P = $P(x,y)$ on $TM \backslash \{0\}$ such that 
\begin{equation}\label{P.3.17}
G^i = \bar{G}^i + P y^i.
\end{equation}
Since $\bar{b}_{i|j}$= 0, implies $\bar{s}^i_0 =0$, from (\ref{P.3.16}), we get
\begin{equation}\label{P.3.17.1}
\bar{G}^i = G^i_{\bar{\alpha}}.
\end{equation}
Putting (\ref{P.3.15}), (\ref{P.3.17.1}) in (\ref{P.3.17}), we have 
\begin{eqnarray*}
G^i_{\alpha}= G^i_{\bar{\alpha}} - \frac{(\epsilon \alpha^3 - 3 \epsilon k \alpha \beta^2 - 4k^2 \beta^3)}{\alpha^2 + \epsilon \alpha \beta + k\beta^2}\tau y^i - 2k\tau \alpha^2 b^i + Py^i,
\end{eqnarray*}
which implies
\begin{equation}\label{P.3.18}
\Bigl(P - \frac{(\epsilon \alpha^3 - 3 \epsilon k \alpha \beta^2 - 4k^2 \beta^3)}{\alpha^2 + \epsilon \alpha \beta + k\beta^2}\tau \Bigr)y^i  = G^i_{\alpha} - G^i_{\bar{\alpha}} + 2k\tau \alpha^2 b^i.
\end{equation}
Note that the right side of (\ref{P.3.18}) is a quadratic in y. Then there exists a 1- form $\theta = \theta_i(x) y^i$ on M, such that
\begin{eqnarray}
P - \frac{(\epsilon \alpha^3 - 3 \epsilon k \alpha \beta^2 - 4k^2 \beta^3)}{\alpha^2 + \epsilon \alpha \beta + k\beta^2}\tau = \theta,\\
\end{eqnarray} 
which implies that
\begin{equation*}
G^i_{\alpha} - G^i_{\bar{\alpha}} + 2k\tau \alpha^2 b^i = \theta y^i,
\end{equation*}
or
\begin{equation}\label{P.3.19}
G^i_{\alpha} = G^i_{\bar{\alpha}} + \theta y^i -2k\tau \alpha^2 b^i.
\end{equation}
This completes the proof of necessity.\\
\\
Conversely, from (\ref{P.3.15}), (\ref{P.3.16}) and (\ref{P.given}), we have
\begin{equation}
G^i = G^i_{\bar{\alpha}} + \Bigl(\theta + \frac{(\epsilon \alpha^3 - 3 \epsilon k \alpha \beta^2 - 4k^2 \beta^3)}{\alpha^2 + \epsilon \alpha \beta + k\beta^2}\tau\Bigr)y^i.
\end{equation}
Thus F is projectively equivalent to $\bar{F}$. This completes the proof of the Theorem \ref{P.them.1}.
\end{proof}

\section{Isotropic Berwald Curvature and S-curvature}
The Berwald tensor of Finsler metric F with the spray coefficients $G^i$ is defined by \textbf{B}$_y$ $= B^i_{jkl} dx^j \otimes \partial_i \otimes dx^k \otimes dx^l$,
 where $B^i_{jkl} = \frac{\partial^3 G^i}{\partial y^j \partial y^k \partial y^l}$. A Finsler metric F is of isotropic Berwald curvature if
\begin{equation}
B^i_{jkl} = c(F_{y^jy^k} \delta^i_l + F_{y^k y^l}\delta^i_j + F_{y^j y^k y^l y^i}),
\end{equation}
where c = $c(x)$ is a scalar function on M [\ref{PC-33}]. The mean Berwald tensor \textbf{E}$_y$ = $E_{ij}dx^i \otimes
dx^j$ is defined by
\begin{equation}
E_{ij} = \frac{1}{2} B^m_{mij} = \frac{1}{2} \frac{\partial^2}{\partial y^i \partial y^j}\Bigl(\frac{\partial G^m}{\partial y^m}\Bigr).
\end{equation}
A Finsler metric $F(x)$ is of isotropic mean Berwald curvature if
\begin{equation}
E_{ij} = \frac{(n+1)}{2} c F_{y^iy^j},
\end{equation}
where c=$c(x)$ is a scalar function on M. Clearly, the Finsler metric of isotropic Berwald curvature must be of isotropic mean Berwald curvature [\ref{PC-12.12}].\\

A Finsler metric F is said to have isotropic S-curvature, if \textbf{S} $= (n + 1) c(x) F$, for some scalar function $c(x)$ on M [\ref{PC-33}].

\begin{lemma} ([\ref{PC-4.4}]) \label{PC-L.3.1}
For an $(\alpha, \beta)$-metric $F = \alpha + \epsilon \beta + k\frac{\beta^2}{\alpha}$ on an n-dimensional manifold M, where $\epsilon$ and $k\neq 0$ are constants, the following are equivalent:\\
\ \ (a).\ F has isotropic S-curvature, i. e., \textbf{S} $= (n+1) c F$;\\
\ \ (b).\ F is of isotropic mean Berwald curvature, $E_{ij} = \frac{(n+1)}{2} c F_{y^i y^j}$;\\
\ \ (c).\ F is a weak Berwald metric, i. e., \textbf{E} = 0;\\
\ \ (d).\ F has vanished S-curvature, i. e., \textbf{S} = 0;\\
\ \ (e).\ $\beta$ is  a killing 1-form of constant length with respect to $\alpha$, i. e., $r_{00} = s_0 = 0$, where $c\ =\ c(x)$ is a scalar function.
\end{lemma}

\begin{lemma} ([\ref{PC-2.2}]) \label{PC-L.3.2}
For a Matsumoto metric F =$\frac{\alpha^2}{\alpha - \beta}$ on an n-dimensional manifold M, the following are equivalent:\\
\ \ (a).\ F has isotropic S-curvature, i. e., \textbf{S} $= (n+1) c F$;\\
\ \ (b).\ F is of isotropic mean Berwald curvature, $E_{ij} = \frac{(n+1)}{2} c F_{y^i y^j}$;\\
\ \ (c).\ F is a weak Berwald metric, i. e., \textbf{E} = 0;\\
\ \ (d).\ F has vanished S-curvature, i. e., \textbf{S} = 0;\\
\ \ (e).\ $\beta$ is  a killing 1-form of constant length with respect to $\alpha$, i. e., $r_{00} = s_0 = 0$, where $c\ =\ c(x)$ is a scalar function.
\end{lemma}

\begin{theorem}
Let F = $\alpha + \epsilon \beta + k\frac{\beta^2}{\alpha}$ ($\epsilon$ and k $\neq$ 0 are constants) be projectively equivalent to $\bar{F} = \frac{\bar{\alpha}^2}{\bar{\alpha} - \bar{\beta}}$ and $\bar{F}$ have isotropic Berwald curvature. Then F has isotropic Berwald curvature if and only if F has isotropic S-curvature.
\end{theorem}
\begin{proof}
Suppose F has isotropic Berwald curvature, then F has isotropic mean Berwald curvature.\\
By Lemma \ref{PC-L.3.1}, F is of isotropic S- curvature. The necessity is proved.\\
We now show the sufficiency.\\
Since F and $\bar{F}$ are projectively equivalent, (\ref{PC.10.10}) holds.\\
Suppose that F has isotropic S-curvature, i. e., 
\begin{equation*}
S = (n + 1) c(x) F.
\end{equation*}
From Lemma \ref{PC-L.3.2}, we get $$E_{ij} = \Bigl(\frac{n+1}{2}\Bigr)c F_{y^iy^j}.$$
Since $\bar{F}$ has isotropic Berwald curvature, we get
\begin{equation}
\bar{B}^i_{jkl} = \bar{c} \{\bar{F}_{y^jy^k}\delta^i_l + \bar{F}_{y^jy^l}\delta^i_k +\bar{F}_{y^ky^l}\delta^i_j + \bar{F}_{y^jy^ky^l}y^i \},
\end{equation}
where $\bar{c} = \bar{c}(x)$ is a scalar function on M. Hence, by the definition of the mean Berwald tensor, it follows from (\ref{PC.10.10}) that
\begin{equation*}
cF_{y^iy^j} = \bar{c}\bar{F}_{y^iy^j} + P_{y^iy^j},
\end{equation*}
which yields that 
\begin{equation*}
cF_{y^iy^jy^k} = \bar{c}\bar{F}_{y^iy^jy^k} + P_{y^iy^jy^k}.
\end{equation*}
We now have
\begin{eqnarray*}
B^i_{jkl} &=& \frac{\partial^3 G^i}{\partial y^j \partial y^k \partial y^l}\\
&=& \bar{B}^i_{jkl} + (P_{y^jy^k} \delta^i_l + P_{y^jy^l} \delta^i_k + P_{y^ky^l} \delta^i_j + P_{y^jy^ky^l} y^i)\\
&=& \bar{c} (\bar{F}_{y^jy^k}\delta^i_l + \bar{F}_{y^jy^l}\delta^i_k +\bar{F}_{y^ky^l}\delta^i_j+\bar{F}_{y^jy^ky^l}y^i )\\
&&+ (P_{y^jy^k}\delta^i_l + P_{y^jy^l}\delta^i_k + P_{y^ky^l}\delta^i_j + P_{y^jy^ky^l}y^i )\\
&=& c (F_{y^jy^k}\delta^i_l + F_{y^jy^l}\delta^i_k + F_{y^ky^l}\delta^i_j+ F_{y^jy^ky^l}y^i ).\\
\end{eqnarray*}
which means that F has isotropic Berwald curvature. We complete the proof.
\end{proof}
We could obtain the following theorem by the above methods.
\begin{theorem}
Let F = $\alpha + \epsilon \beta + k \frac{\beta^2}{\alpha}$ ($\epsilon$ and k $\neq 0$ are constants) be projectively equivalent to $\bar{F} = \frac{\bar{\alpha}^2}{\bar{\alpha} - \bar{\beta}}$ and F has isotropic Berwald curvature. Then $\bar{F}$ has isotropic Berwald curvature if and only if $\bar{F}$ has isotropic S-curvature.
\end{theorem}
For an $(\alpha, \beta)$- metric $F = \alpha + \epsilon \beta + k\frac{\beta^2}{\alpha}$ of isotropic S-curvature, we have the following:
\begin{theorem}
Let $F = \alpha + \epsilon \beta + k\frac{\beta^2}{\alpha}$ ($\epsilon$ and k $\neq 0$ are constants) be an $(\alpha, \beta)$-metric and $\bar{F} = \frac{\bar{\alpha}^2}{\bar{\alpha} - \bar{\beta}}$ be a Matsumoto metric on an n-dimensional manifold M ($n\ \geq\ 3$), where $\alpha$ and $\bar{\alpha}$ are two Riemannian metrics, $\beta = b_iy^i$ and $\bar{\beta}$ are two non zero 1-forms. Assume that F has isotropic S- curvature. Then F is projectively equivalent to $\bar{F}$ if and only if $\bar{\beta }$ is closed and the following conditions hold:\\
\ \ (a).\ $\alpha$ is projectively equivalent to $\bar{\alpha}$;\\
\ \ (b).\ $\beta$ is parallel with respect to $\alpha$, i. e., $b_{i |j} = 0$.\\
\end{theorem}
\begin{proof}
The sufficiency is obvious from Theorem \ref{P.them.1}.\\
Now, for the necessity, from Theorem \ref{P.them.1}, we have that if F is projectively equivalent to $\bar{F}$, then
\begin{equation}\label{P.3.34}
b_{i|j} = 2 \tau \{(1 + 2 kb^2)a_{ij} - 3kb_ib_j\}.
\end{equation}
Contracting (\ref{P.3.34}) with $y^i$ and $y^j$ yields, 
\begin{equation}
r_{00} = 2 \tau \{(1+ 2b^2) \alpha^2 - 3 k\beta^2\}.
\end{equation}
By Lemma \ref{PC-L.3.2}, if F has isotropic S-curvature, then $r_{00}$ = 0, i. e.,
\begin{equation}
r_{00} = 2 \tau \{(1 + 2b^2)\alpha^2 - 3 k \beta^2\} = 0.
\end{equation}
If possible let $\tau\ \neq\ 0$, then 
\begin{equation}
(1 + 2kb^2) \alpha^2 = 3 k \beta^2.
\end{equation}
Since $\alpha^2$ does not contain $\beta$, we have $(1 + 2kb^2) = k= 0$ which is impossible. Thus $\tau = 0$. By Theorem \ref{P.them.1}, the necessity is obvious. We complete the proof. \\
\end{proof}
Since the Finsler metric of isotropic Berwald curvature must be of isotropic mean Berwald curvature. Thus, by Lemma \ref{PC-L.3.2}, assuming that F has isotropic Berwald curvature, we immediately obtain.\\
\begin{theorem}
Let $F = \alpha + \epsilon \beta + k\frac{\beta^2}{\alpha}$ ($\epsilon$ and k $\neq 0$ are constants) be an $(\alpha, \beta)$-metric and $\bar{F} = \frac{\bar{\alpha}^2}{\bar{\alpha} - \bar{\beta}}$ be a Matsumoto metric on an n-dimensional manifold M ($n\ \geq\ 3$), where $\alpha$ and $\bar{\alpha}$ are two Riemannian metrics, $\beta = b_iy^i$ and $\bar{\beta}$ are two non zero 1-forms. Assume that F has isotropic Berwald curvature. Then F is projectively equivalent to $\bar{F}$ if and only if $\bar{\beta }$ is closed and the following conditions hold:\\
\ \ (a).\ $\alpha$ is projectively equivalent to $\bar{\alpha}$;\\
\ \ (b).\ $\beta$ is parallel with respect to $\alpha$, i. e., $b_{i |j} = 0$.\\
\end{theorem}

\end{document}